\documentclass[a4paper,12pt]{article}
\usepackage{amsmath, amsthm, amstext}
\usepackage{amssymb, array, amsfonts}

\numberwithin{equation}{section}

\def \al{\alpha}
\def \be{\beta}
\def \ga{\gamma}
\def \de{\delta}
\def \er{\varepsilon}

\def \ka{\varkappa}
\def \la{\lambda}
\def \si{\sigma}

\def \D{\Delta}

\def \N{\mathbb{N}}

\def \R{\mathbb{R}}

\def \T{\mathbb{T}}
\def\n{\nabla}
\def\dd{\partial}
\def\div{\operatorname{div}}

\def\1{1\!\!\!\!1}

\theoremstyle{plain}
\newtheorem{theorem}{\bf Theorem}[section]
\newtheorem{lemma}[theorem]{\bf Lemma}
\newtheorem{cor}[theorem]{\bf Corollary}
\newtheorem{defi}[theorem]{\bf Definition}

\theoremstyle{remark}
\newtheorem{rem}[theorem]{\bf Remark}

\renewcommand{\le}{\leqslant}
\renewcommand{\ge}{\geqslant}

\renewcommand{\qed}{\vrule height7pt width5pt depth0pt}

\begin{document}

\title{Uniqueness of the Leray-Hopf solution for a dyadic model}

\author{N.~D.~Filonov}

\date{}

\maketitle

\begin{abstract}
The dyadic problem
$\dot u_n + \la^{2n} u_n - \la^{\be n} u_{n-1}^2 + 
\la^{\be(n+1)} u_n u_{n+1} = 0$ 
with "smooth" initial data is considered.
The uniqueness of the Leray-Hopf solution is proved.
\end{abstract}

\section*{Introduction}
\subsection{Analogy with the Navier-Stokes equations}
We consider the following problem
\begin{equation}
\label{01}
\begin{cases}
\dot u_n (t) + \la^{2n} u_n (t) - \la^{\be n} u_{n-1} (t)^2 + 
\la^{\be(n+1)} u_n (t) u_{n+1} (t) = 0, \quad t \in [0, \infty), \\
u_n (0) = a_n, \qquad n = 1, 2, \dots .
\end{cases}
\end{equation}
Here $u_0 \equiv 0$, $\la > 1$, $\be > 0$, $a = \{a_n\} \in l_2$.
Last decade, many authors pay attention to the problems of such kind,
see \cite{BFM, BM, BMR, Ch, FP, KP02, KP05, KZ, T, W}.
An important feature of the system \eqref{01} is that it is similar 
to the system of the Navier-Stokes equations
\begin{equation}
\label{02}
\begin{cases}
\dd_t u - \D u + P \left((u, \n) u\right) = 0 
\quad \text{in}\ \  [0, \infty) \times \T^d, \\
\div u = 0, \left.u\right|_{t=0} = a(x) .
\end{cases}
\end{equation}
Here $\T^d$ is a $d$-dimensional torus, 
and $P$ is the orthoprojector in $L_2(\T^d)$ onto 
the subspace of divergence-free functions.
Both systems can be written in the following abstract way:
\begin{equation*}
\begin{cases}
\dot u + A u + B (u, u) = 0, \quad t \in [0, \infty) , \\
u (0) = a .
\end{cases}
\end{equation*}
Here the function $u(t)$ takes values in a Hilbert space ${\mathcal H}$,
${\mathcal H} = l_2$ in the case \eqref{01}, and 
${\mathcal H} = \{u \in L_2 (\T^d) : \div u = 0\}$ in the case \eqref{02}.
$A$ is a self-adjoint non-negative unbounded operator in ${\mathcal H}$.
$B$ is a bilinear unbounded map
$B : {\mathcal H} \times {\mathcal H} \to {\mathcal H}$,
having two principal properties:
\begin{itemize}
\item the cancellation property
$$
\left(B(u,u), u\right)_{\mathcal H} = 0
$$
for the dense set in ${\mathcal H}$ of "good" $u$;
\item an estimate
$$\|B(u,u)\|_{\mathcal H} \le 
C \|A^{\si_1} u\|_{\mathcal H} \|A^{\si_1} u\|_{\mathcal H} 
$$
holds with some $\si_1$, $\si_2$. 
\end{itemize}
The orders $\si_1$, $\si_2$ can vary, 
but the sum $(\si_1 + \si_2)$ is fixed.
For the case of the problem \eqref{01} we have
$\si_1 + \si_2 = \frac\be2$, 
and for the case \eqref{02} the imbedding theorems imply
$\si_1 + \si_2 = \frac{d+2}4$, where $d$ is a dimension
of the space variables.
Thus, the most interesting value of the parameter $\be$ in \eqref{01}
is $\be = 5/2$.

We consider \eqref{01} as a toy model for the Navier-Stokes equation.
Note also, that the system \eqref{01} was written firstly in \cite{DN} 
as a model describing the turbulence flow in hydrodynamics.

The next system
\begin{equation}
\label{03}
\begin{cases}
\dot u_n (t) - \ka^n u_{n-1} (t)^2 + 
\ka^{n+1} u_n (t) u_{n+1} (t) = 0, \quad t \in [0, \infty), \\
u_n (0) = a_n, \qquad n = 1, 2, \dots ,
\end{cases}
\end{equation}
is also actively studied. 
It can be obtained from \eqref{01} by elimination 
of the linear term $\la^{2n} u_n (t)$; $\ka = \la^\be > 1$.
The system \eqref{03} is a model for the Euler equations
in hydrodynamics (which is the problem \eqref{02} 
without the viscous term $\D u$).

\subsection{Types of solutions}

\begin{defi}
\label{o01}
A sequence of functions $\{u_n\}$ is called a weak solution 
to the problem \eqref{01} if $u_n \in C^\infty [0, \infty)$ for all $n$,
and \eqref{01} is fulfilled.
\end{defi}

\begin{defi}
\label{o02}
A weak solution $\{u_n\}$ is called a Leray-Hopf solution 
to the problem \eqref{01} if the inequality
\begin{equation}
\label{04}
\sum_{n=1}^\infty \left( u_n (t)^2 +
2 \int_0^t \la^{2n} u_n (\tau)^2 d\tau \right) 
\le \sum_{n=1}^\infty a_n^2 
\quad \forall t \in [0, \infty)
\end{equation}
holds.
In the same manner one defines a Leray-Hopf solution
on a finite interval $[0,T]$.
\end{defi}

It is easy to show that a Leray-Hopf solution exists
for any initial data $\{a_n\} \in l_2$, 
see Theorem \ref{t22} below.

We will not introduce a formal definition of a strong solution
to \eqref{01}.
But we will use the words "strong solution" in a non-formal manner,
meaning that the functions $u_n (t)$ decrease fast as $n \to \infty$
in a uniform or in an integral norm.
And we would like to warn the reader that the sense may be different 
in different situations.
For example, it is easy to see that if the solution $\{u_n\}$
is strong in such a way that
$$
\int_0^T |u_n (t)|^3 dt = o (\la^{-\be n}), \quad n \to \infty,
$$
then it is also a Leray-Hopf solution on $[0,T]$.

\subsection{Formulation of the results}
A.~Cheskidov proved that the system \eqref{01} 
has a global in time strong solution if $\be \le 2$,
and that a blow-up takes place if $\be > 3$.
Let us formulate these results more precisely.

\begin{theorem}[\cite{Ch}]
\label{t03}
Let $\be \le 2$, $\sum_{n=1}^\infty \la^{2n} a_n^2 < \infty$.
Then the problem \eqref{01} has a solution $\{u_n\}$ satisfying
the relation
$$
\sup_{t} \sum_{n=1}^\infty \la^{2n} u_n (t)^2 < \infty .
$$
\end{theorem}

\begin{theorem}[\cite{Ch}]
\label{t04}
Let $\be > 3$, $\er > 0$.
Then any solution $\{u_n\}$ to \eqref{01} will have
$$
\int_0^T \left(\sum_{n=1}^\infty \la^{2(\er+1/3)\be n} 
u_n (t)^2\right)^{3/2} dt = + \infty
$$
for some $T$ whenever all $a_n \ge 0$ and 
$\sum_{n=1}^\infty \la^{2\be\er n} a_n^2 > M$.
Here $M$ is a constant depending on $\er$.
For example, one can take $a_n = 0$ for $n \ge 2$, if $a_1$ is big enough.
\end{theorem}

D.~Barbato, F.~Morandin and M.~Romito have considered 
the important particular case for the problem \eqref{01}
when the initial data are non-negative.

\begin{theorem}[\cite{BMR}]
\footnote{The value of the parameter $\be$ 
represents in a sense a "rate of domination" of 
the convective (non-linear) term in the equations 
over the dissipative (linear) term.
One can associate such a parameter in an equivalent way
with the convective term or with the dissipative one.
In \cite{BMR} the parameter is associated with the convective term,
like in \eqref{01}.
In \cite{Ch} it is associated with the linear term:
$$
\dot u_n + \la^{2\al n} u_n - \la^n u_{n-1}^2 + 
\la^{n+1} u_n u_{n+1} = 0.
$$
The currency exchange is as follows:
$$
\la_{Ch} = \la_{BMR}^\be, \quad \la_{BMR} = \la_{Ch}^\al, 
\quad \al \be = 1 .
$$
}
\label{t05}
Let $\la = 2$, $2 < \be \le \frac52$.
Let $\{a_n\} \in l_2$, $a_n \ge 0$.
Then there exists a unique weak solution $\{u_n\}$ 
to the problem \eqref{01}, and
$$
u_n (t) = O (\la^{-\ga n}), \ \ n \to \infty, 
\quad \forall \ga, \ \ \forall t > 0 .
$$
\end{theorem}

Of course, one would like to remove 
1) the condition of non-negativity for the initial data; 
and
2) the fixation of the value $\la = 2$.
We are unable to prove the existence of a strong solution
in such setting. 
Instead of it we prove the uniqueness of a Leray-Hopf solution.
We use heavily the following result of Barbato and Morandin on 
the uniqueness of a weak solution to the problem \eqref{03} 
if the initial data are non-negative.

\begin{theorem}[\cite{BM}]
\label{t06} 
Let $\ka > 1$, $\{a_n\} \in l_2$, $a_n \ge 0$.
Then there exists a unique weak solution to the problem \eqref{03}.
\end{theorem}

This result can be extended to the case of the system \eqref{01},
see Theorem \ref{t42} below.
Using this fact, we establish the uniqueness of the Leray-Hopf solution
to \eqref{01} for sufficiently good initial data.
We formulate now the main result of the paper.

\begin{theorem}
\label{t07}
Let $\la > 1$, $\be > 0$.
Assume that $\{a_n\} \in l_2$ if $\be \le 2$, and
$a_n = o (\la^{(2-\be)n})$, $n \to \infty$, if $\be > 2$.
Then the problem \eqref{01} has a unique Leray-Hopf solution.
\end{theorem}

\begin{rem}
It will be seen from the proof that 
we establish a slightely more strong result.
Namely, for $\be > 2$ the following restriction 
on the initial data is sufficient for the claim:
\begin{equation}
\label{05}
\limsup_{n\to\infty} |a_n| \la^{(\be-2)n} < \er : = 
\min \left(\frac1{\la^2+\la^{\be-1}}, \frac1{\la^2+\la^{2\be-4}}\right) .
\end{equation}
\end{rem}

It is easy to show that if there exists a strong solution to \eqref{01},
then it is unique in the class of the Leray-Hopf solutions,
see Theorem \ref{t14} below.
For $\be \le 3$, we do not know if the Leray-Hopf solution 
guaranteed by Theorem \ref{t07} is a strong one.
If $\be > 3$ then it can be {\it not strong} solution, 
see Theorem \ref{t04}.

\subsection{Stationary solutions}
We call a stationary solution a sequence $\{a_n\}$ of real numbers
satisfying the stationary equations
\begin{equation}
\label{06}
\la^{2n} a_n - \la^{\be n} a_{n-1}^2 + 
\la^{\be(n+1)} a_n a_{n+1} = 0, \quad a_0 = 0, \quad n = 1, 2, \dots .
\end{equation}

\begin{theorem}[\cite{BMR}]
\label{t09}
a) Let a sequence $\{a_n\}$ satisfy \eqref{06}.
Then if $a_{n_0} = 0$ then $a_n = 0$ for all $n \le n_0$.
If $a_{n_1} \neq 0$, then 
$$
a_n \le - \la^{(2-\be)n - 2} \quad \text{for all} \ n > n_1 .
$$

b) Let the parameters $\be \in (2,3)$ and $\la$ be such that
$\la^{2\be-6} < \frac13$.
Then there exists a non-trivial solution $\{a_n\}$ to \eqref{06}
such that
$$
a_n = O \left(\la^{(2-\be)n}\right), \ \ n \to \infty.
$$
\end{theorem}

The claim a) of this Theorem means that if $\be \le 2$
then a non-tirivial stationary solution can not belong to $l_2$.

We establish the existence of stationary solution for all
values of the parameters $\la$ and $\be$.

\begin{theorem}
\label{t010}
There exists a non-trivial solution $\{a_n\}$ to \eqref{06}
such that
\begin{itemize}
\item $a_n = O \left(\la^{(2-\be)n}\right)$, $n \to \infty$, 
if $\be < 3$,
\item $a_n = O \left(n\la^{-n}\right)$, $n \to \infty$, 
if $\be = 3$,
\item $a_n = O \left(\la^{-\be n/3}\right)$, $n \to \infty$, 
if $\be > 3$.
\end{itemize}
These estimates can not be improved, 
as any non-trivial solution $\{a_n\}$ to \eqref{06}
satisfies the following relation:
\begin{itemize}
\item $\limsup_{n\to\infty} |a_n| \la^{(\be-2)n} > 0$
if $\be < 3$, 
\item $\limsup_{n\to\infty} |a_n| \la^n = +\infty$
if $\be = 3$, 
\item $\limsup_{n\to\infty} |a_n| \la^{\be n/3} > 0$
if $\be > 3$.
\end{itemize}
\end{theorem}

It is clear that a stationary solution is not a Leray-Hopf solution.
On the other hand, a Leray-Hopf solution exists for
any initial data from $l_2$.
So, these results mean that a weak solution to \eqref{01}
may be non-unique if $\be > 2$.
Note, that for $\be \in (2,3)$ we have this non-uniqueness
of weak solution with initial data of order 
$a_n \sim \la^{(2-\be)n}$ (cf. Theorem \ref{t07}).

Note also, that stationary solutions to \eqref{01}
are related to the self-similar solutions to the system \eqref{03}
with a blow-up, see \cite{BFM}.
Theorem \ref{t010} in the case $\be > 3$ is essentially proved in \cite{BFM},
see Remark \ref{r522} below.

\subsection{Structure of the paper}
In the first section we prove the basic properties of weak solutions
and the uniqueness of the strong solution.
In the second section we introduce the notion of Galerkin solution,
and describe its properties for "good" initial data.
In the third section we estimate the integral
$\int_0^\infty u_n^3 dt$ for a non-negative case.
In the fourth section we use this estimate to justify 
the uniqueness of weak solution in the non-negative case,
and then we prove Theorem \ref{t07}.
In the last section we study the steady-state solutions
and prove Theorem \ref{t010}.

\subsection{Acknowledgements}
The author thanks T.~Shilkin for fruitful discussions.

%%%%%%%%%%%%%%%%%%%%%%%%%%%%%%%%%%%%%%%%%%%%%%%%%%%%%%%%%%%%%%%%%%%%%%%%%

\section{Auxiliary results}
\subsection{Properties of weak solutions}
The content of this subsection is borrowed from the paper \cite{Ch}.
We give the proofs for the sake of completeness.

\begin{lemma}[\cite{Ch}]
\label{l11}
Let $\{u_n\}$ be a weak solution to \eqref{01}.

a) If $a_n \ge 0$ then $u_n(t) \ge 0$ for all $t$.

b) If $a_n < 0$ then either $u_n(t) < 0$ for all $t$,
or there exists a time $\tau_n$ such that
$$
u_n(t)<0 \ \text{for} \ t<\tau_n, \quad
u_n(\tau_n)=0, \quad
u_n(t) \ge 0 \ \text{for} \ t>\tau_n .
$$
\end{lemma}

\begin{proof}
One can rewrite the differential equation for $u_n$
as an integral one,
\begin{eqnarray}
\label{11}
u_n(t) = a_n 
\exp\left(-\int_0^t (\la^{2n} + \la^{\be(n+1)} u_{n+1}(s)) ds\right) \\
+ \int_0^t \exp\left(-\int_s^t (\la^{2n} 
+ \la^{\be(n+1)} u_{n+1}(\si)) d\si\right)
\la^{\be n} u_{n-1}(s)^2 ds .
\nonumber
\end{eqnarray}
Therefore,
$$
u_n(t) \ge a_n 
\exp\left(-\int_0^t (\la^{2n} + \la^{\be(n+1)} u_{n+1}(s)) ds\right),
$$
and a) follows. 
Clearly, a) implies b).
\end{proof}

We introduce the notation
\begin{equation}
\label{12}
l_2^+ = \left\{ \{a_n\} \in l_2 : \ \text{there is a number}\ M 
\ \text{such that} \ a_n \ge 0 \ \text{for} \ n \ge M \right\} .
\end{equation}

\begin{lemma}[\cite{Ch}]
\label{l12}
Let $\{u_n\}$ be a weak solution to \eqref{01} with $\{a_n\} \in l_2^+$.
Then $\{u_n\}$ is a Leray-Hopf solution.
\end{lemma}

\begin{proof}
Let $a_n \ge 0$ for $n \ge M$.
By virtue of Lemma \ref{l11} $u_n(t) \ge 0$ for $n \ge M$.
Now, \eqref{01} implies for $N \ge M$
$$
\sum_{n=1}^N \left(u_n \dot u_n + \la^{2n} u_n^2 \right)  
= - \la^{\be(N+1)} u_N^2 u_{N+1} \le 0,
$$
and therefore,
$$
\sum_{n=1}^N \left( \frac{u_n(t)^2}2 +
\int_0^t \la^{2n} u_n (\tau)^2 d\tau \right) 
\le \frac12 \sum_{n=1}^N a_n^2 
\le \frac12 \sum_{n=1}^\infty a_n^2 .
$$
Passing to the limit $N\to\infty$, we get \eqref{04}.
\end{proof}

It is clear from the proof that the following property 
of initial data is sufficient:
there is a sequence $n_k \to \infty$ such that $a_{n_k} \ge 0$.
Note also, that a non-trivial stationary solution
(which is not Leray-Hopf) has all $a_n < 0$ 
beginning from some number $n_1$ due to Theorem \ref{t09}.

\subsection{Uniqueness of the strong solution}

\begin{lemma}
\label{l13}
Assume that the functions $\{u_n\}$, $\{v_n\}$ and $\{w_n\}$
satisfy the inequalities
$$
|u_n(t)| \le C \la^{(2-\be)n}, \quad t \in [0,T],
$$
$$
\sum_{n=1}^\infty \la^{2n} \int_0^T 
\left(v_n(t)^2 + w_n(t)^2\right) dt < \infty .
$$
Then the series
$$
\sum_{n=1}^\infty \la^{\be n} \int_0^T 
|u_n(t)| \left(|v_{n-1}(t)| + |v_n(t)| + |v_{n+1}(t)|\right)
\left(|w_{n-1}(t)| + |w_n(t)| + |w_{n+1}(t)|\right) dt
$$
converges.
\end{lemma}

This Lemma is evident.

\begin{theorem}
\label{t14}
Let $\{u_n\}$ and $\{v_n\}$ be the Leray-Hopf solutions to \eqref{01}
corresponding to the same initial data $\{a_n\} \in l_2$.
Assume that there is a number $L$ such that
$$
|u_n(t)| \le \er_1 \la^{(2-\be)n}, \ \text{for}\ n \ge L,\quad t \in [0,T],
$$
where
$\er_1 = \left(\la^2+\la^{\be -1}\right)^{-1}$.
Then $u_n \equiv v_n$ on $[0,T]$.
\end{theorem}

\begin{proof}
We have
$$
\frac{d}{dt}(u_n v_n) + 2 \la^{2n} u_n v_n =
\la^{\be n} \left(u_n v_{n-1}^2 + u_{n-1}^2 v_n\right)
- \la^{\be(n+1)} \left(u_n v_n v_{n+1} + u_n u_{n+1} v_n\right) .
$$
Therefore,
\begin{eqnarray*}
u_n (t) v_n (t) - a_n^2 + 2 \int_0^t \la^{2n} u_n v_n d\tau \\
= \int_0^t \left(\la^{\be n} \left(u_n v_{n-1}^2 + u_{n-1}^2 v_n\right)
- \la^{\be(n+1)} \left(u_n v_n v_{n+1} + u_n u_{n+1} v_n\right)\right) d\tau,
\end{eqnarray*}
and
\begin{eqnarray}
\nonumber
\sum_{n=1}^\infty 
\left(u_n (t) v_n (t) - a_n^2 + 2 \int_0^t \la^{2n} u_n v_n d\tau\right) \\
\label{13}
= \sum_{n=1}^\infty \int_0^t \la^{\be(n+1)}
\left(u_{n+1} v_n^2 + u_n^2 v_{n+1} - 
u_n v_n v_{n+1} - u_n u_{n+1} v_n\right) d\tau \\
= \sum_{n=1}^\infty \int_0^t \la^{\be(n+1)}
(u_n - v_n) \left(u_{n+1} (u_n - v_n) - 
u_n (u_{n+1} - v_{n+1}) \right) d\tau .
\nonumber
\end{eqnarray}
All series in the last equality converge due to Lemma \ref{l13}.

Let us estimate the remainder of the last series.
We have
$$
\left|\sum_{n=L}^\infty \int_0^t \la^{\be(n+1)}
u_{n+1} (u_n - v_n)^2 d\tau\right| 
\le \sum_{n=L}^\infty \int_0^t \er_1 \la^{2n+2} (u_n - v_n)^2 d\tau ,
$$
and
\begin{eqnarray*}
\left|\sum_{n=L}^\infty \int_0^t \la^{\be(n+1)}
u_n (u_n - v_n) (u_{n+1} - v_{n+1}) d\tau \right| \\
\le \sum_{n=L}^\infty \int_0^t \er_1 \la^{2n+\be}
\left(\frac1{2\la} (u_n - v_n)^2 
+ \frac\la2 (u_{n+1} - v_{n+1})^2 \right) d\tau 
\le \sum_{n=L}^\infty \int_0^t \er_1 \la^{2n-1+\be} (u_n - v_n)^2 d\tau .
\end{eqnarray*}
Now, taking into account the definition of $\er_1$ 
and the equality \eqref{13}, we get
\begin{eqnarray}
\label{14}
- 2 \sum_{n=1}^\infty 
\left(u_n (t) v_n (t) + 2 \int_0^t \la^{2n} u_n v_n d\tau\right) \\
\le - 2 \sum_{n=1}^\infty a_n^2 
+ C_1 \sum_{n=1}^L \int_0^t (u_n-v_n)^2 d\tau
+ 2 \sum_{n=L}^\infty \int_0^t \la^{2n} (u_n-v_n)^2 d\tau ,
\nonumber
\end{eqnarray}
where the constant 
$$
C_1 = 4 \la^{\be L} \sup_{n = 1, \dots, L} \|u_n\|_{L_\infty} 
\le 4 \la^{\be L} \|a\|_{l_2}
$$
does not depend on $t$.

Recall that
\begin{equation}
\label{15}
\sum_{n=1}^\infty \left( u_n (t)^2 +
2 \int_0^t \la^{2n} u_n (\tau)^2 d\tau \right) 
\le \sum_{n=1}^\infty a_n^2 ,
\end{equation}
\begin{equation}
\label{16}
\sum_{n=1}^\infty \left( v_n (t)^2 +
2 \int_0^t \la^{2n} v_n (\tau)^2 d\tau \right) 
\le \sum_{n=1}^\infty a_n^2 ,
\end{equation}
by definition of the Leray-Hopf solutions.
Summarizing \eqref{14} -- \eqref{16} we obtain
\begin{eqnarray*}
\sum_{n=1}^\infty \left( (u_n (t) - v_n (t))^2 +
2 \int_0^t \la^{2n} (u_n (\tau) - v_n (\tau))^2 d\tau \right) \\
\le C_1 \sum_{n=1}^L \int_0^t (u_n(\tau) - v_n(\tau))^2 d\tau
+ 2 \sum_{n=L}^\infty \int_0^t \la^{2n} (u_n(\tau) - v_n(\tau))^2 d\tau ,
\end{eqnarray*}
where from
$$
\sum_{n=1}^\infty \left( u_n (t) - v_n (t) \right)^2 
\le C_1 \int_0^t \sum_{n=1}^L (u_n(\tau) - v_n(\tau))^2 d\tau .
$$
If we put 
$y(t) = \sum_{n=1}^\infty \left( u_n (t) - v_n (t) \right)^2$
then the last inequality implies
$y(t) \le C_1 \int_0^t y(\tau) d\tau$,
and thus, $y(t) \equiv 0$ and $u_n \equiv v_n$.
\end{proof}

Note, that this proof is nothing but
a simplified version of the standard proof of the uniqueness
of a strong solution among Leray-Hopf solutions in the NSE theory.

\begin{cor}
\label{c15}
Let $\be < 2$.
Then the Leray-Hopf solution is unique for arbitrary initial data
$\{a_n\} \in l_2$.
\end{cor}

\begin{proof}
By definition of Leray-Hopf solution
we have $|u_n(t)| \le \|a\|_{l_2}$ for all $n$ and $t$.
Therefore, 
$|u_n(t)| \le \er_1 \la^{(2-\be)n}$
for sufficiently large $n$.
\end{proof}

%%%%%%%%%%%%%%%%%%%%%%%%%%%%%%%%%%%%%%%%%%%%%%%%%%%%%%%%%%%%%%%%%%%%%%%%%

\section{Galerkin solutions}
\subsection{Definition}
\label{s21}
Galerkin solutions is an important tool 
for the problem under consideration.
In order to define it let us consider for each $N \in \N$
the following problem
\begin{equation}
\label{21}
\begin{cases}
\dot v_n^{(N)} + \la^{2n} v_n^{(N)} - 
\la^{\be n} \left(v_{n-1}^{(N)}\right)^2 + 
\la^{\be(n+1)} v_n^{(N)} v_{n+1}^{(N)} = 0, \quad t \in [0, \infty), \\
v_n^{(N)} (0) = a_n, \qquad n = 1, \dots, N ; \qquad 
v_0^{(N)} \equiv 0, \ \ v_{N+1}^{(N)} \equiv 0 .
\end{cases}
\end{equation}
It is well known from the ordinary differential equations theory
that the problem \eqref{21} has a unique solution for the small interval of time.
The length of this interval depends on the $l_2$-norm
of the initial data 
$\left(\sum_{n=1}^N a_n^2\right)^{1/2}$ only.
Multiplying the equations \eqref{21} by $v_n^{(N)}$ and 
taking the sum over $n$ we get
$$
\sum_{n=1}^N \left(v_n^{(N)} \dot v_n^{(N)} + 
\la^{2n} \left(v_n^{(N)}\right)^2\right) = 0,
$$
so
\begin{equation}
\label{22}
\frac12 \sum_{n=1}^N \left(v_n^{(N)} (t)^2 - a_n^2\right) 
+ \sum_{n=1}^N \int_0^t \la^{2n} v_n^{(N)}(\tau)^2 d\tau = 0,
\end{equation}
and
\begin{equation}
\label{23}
\sum_{n=1}^N v_n^{(N)} (t)^2 \le \sum_{n=1}^N a_n^2 .
\end{equation}
Last estimate implies that the problem \eqref{21} 
has a global in time solution.
Moreover,
$$
\|v_n^{(N)}\|_{C[0,T]} \le \|a\|_{l_2},
$$
and the equations \eqref{21} yield the boundedness of the sequence
$\{v_n^{(N)}\}_{N=n}^\infty$ in $C^2[0,T]$ for each $n$ and for each $T$.
Therefore, there exists a subsequence $\{v_n^{(N_k)}\}$
converging in $C^1[0,T]$ as $N_k\to\infty$.
Applying the diagonal process,
we obtain a sequence of numbers $\{M_k\}$ such that
$$
v_n^{(M_k)} \underset{k\to\infty}\longrightarrow u_n 
\quad \text{in} \ \ C^1[0,T] \quad
\text{for all} \ n , \quad \text{and for all} \ T .
$$
Clearly, the limit sequence $\{u_n\}$ satisfies \eqref{01}
on $[0, \infty)$.

\begin{defi}
\label{o21}
We call the Galerkin solution to the problem \eqref{01}
a solution $\{u_n(t)\}$ constructed above.
\end{defi}

Note, that the construction of a Galerkin solution
does not imply in general its uniqueness.

Now, let us fix $M\in\N$. 
The relation \eqref{22} yields for $N > M$
$$
\sum_{n=1}^M \left(v_n^{(N)} (t)^2  
+ 2\int_0^t \la^{2n} v_n^{(N)}(\tau)^2 d\tau\right) 
\le \sum_{n=1}^\infty a_n^2 ,
$$
which implies the same inequality for the limit functions
$$
\sum_{n=1}^M \left(u_n (t)^2  
+ 2\int_0^t \la^{2n} u_n (\tau)^2 d\tau\right) 
\le \sum_{n=1}^\infty a_n^2 ,
$$
and therefore, \eqref{04} holds.
Thus, we see that any Galerkin solution to \eqref{01}
is a Leray-Hopf solution as well.
We proved

\begin{theorem}
\label{t22}
Let $\{a_n\} \in l_2$.
Then there is a Leray-Hopf solution to the problem \eqref{01}.
\end{theorem}

Note, that the constructions of this subsection are also standard.

\subsection{Good initial data. 
Estimate from below}

\begin{theorem}
\label{t23}
Let $\er_2 \le \la^{-2}$, and
$a_n \ge -\er_2 \la^{(2-\be)n}$ for $n \ge K$ for some $K$.
Let $\{u_n\}$ be a Galerkin solution to \eqref{01}.
Then
$$
u_n(t) \ge -\er_2 \la^{(2-\be)n} \quad 
\text{for}\  n \ge K, \quad t \in [0, \infty) .
$$
\end{theorem}

\begin{proof}
Let us consider the solution $\{v_n^{(N)}\}$ to \eqref{21}
with some $N > K$.
Similarly to \eqref{11} we have
\begin{equation}
\label{24}
v_n^{(N)} (t) \ge a_n 
\exp\left(-\int_0^t (\la^{2n} + \la^{\be(n+1)} v_{n+1}^{(N)}(s)) ds\right) .
\end{equation}
We show now that
\begin{equation}
\label{25}
v_n^{(N)} (t) \ge -\er_2 \la^{(2-\be)n} \quad 
\text{for}\ n = K, \dots, N+1, \quad t \in [0, \infty), 
\end{equation}
by backward induction in $n$.
For $n=N+1$ the claim is clear as $v_{N+1}^{(N)} \equiv 0$.
Assume that \eqref{25} is true for $n=k+1$.
We have 
$$
\la^{2k} + \la^{\be(k+1)} v_{k+1}^{(N)}(s) 
\ge \la^{2k} (1-\er_2 \la^2) \ge 0
$$
due to the assumption $\er_2 \le \la^{-2}$.
Now, \eqref{24} implies that 
$v_k^{(N)} (t) \ge 0$ for all $t$ if $a_k \ge 0$, and
$$
v_k^{(N)} (t) \ge a_k \ge -\er_2 \la^{(2-\be)k} ,
$$ 
if $a_k < 0$.
Thus, \eqref{25} is true for $n=k$.

So, \eqref{25} is proven.
Passing in \eqref{25} to the limit $N\to\infty$,
we obtain the result. 
\end{proof}

\begin{cor}
\label{c24}
Under the assumptions of Theorem \ref{t23} we have
$$
\la^{2n} + \la^{\be(n+1)} u_{n+1} (t) 
\ge \la^{2n} (1-\er_2 \la^2), \quad 
\text{for}\  n \ge K, \quad t \in [0, \infty) .
$$
\end{cor}

\begin{cor}
\label{c25}
Under the assumptions of Theorem \ref{t23},
if $u_n(t) \le 0$ for some $n\ge K$,
then $\dot u_n(t) \ge 0$.
\end{cor}

\begin{proof}
From the equation \eqref{01} we have
$$
\dot u_n (t) =  \la^{\be n} u_{n-1} (t)^2 
- u_n(t) \left(\la^{2n} + \la^{\be(n+1)} u_{n+1} (t)\right) \ge 0
$$
whenever $u_n(t) \le 0$.
\end{proof}

\subsection{Good initial data.
Estimate of absolute value}

\begin{theorem}
\label{t26}
Let $\er_3 \le \left(\la^2 + \la^{2\be-4}\right)^{-1}$,
and $|a_n| \le \er_3 \la^{(2-\be)n}$ for $n \ge K$ for some $K$.
Let $\{u_n\}$ be a Galerkin solution to \eqref{01}.
Assume that there is a number $L\ge K$ such that
$$
|u_L (t)| \le \er_3 \la^{(2-\be)L} \quad
\text{for all}\ \ t \in [0,T] .
$$
Then
\begin{equation}
\label{26}
|u_n (t)| \le \er_3 \la^{(2-\be)n} \quad
\text{for all}\ \ n \ge L, \quad t \in [0,T] .
\end{equation}
\end{theorem}

\begin{proof}
We proceed by induction in $n$.
The case $n=L$ is given by the assumption.
Assume that \eqref{26} is fulfilled for $n=k-1$.
By virtue of Corollary \ref{c24}
$$
\la^{2k} + \la^{\be(k+1)} u_{k+1} (t) 
\ge \la^{2k} (1-\er_3 \la^2)
\ge \frac{\la^{2k}}{\la^{6-2\be}+1} .
$$
Therefore,
by \eqref{11} and \eqref{26} for $n=k-1$, we have
\begin{eqnarray*}
|u_k(t)| \le |a_k| e^{-\frac{\la^{2k} t}{\la^{6-2\be}+1}} +
\int_0^t e^{-\frac{\la^{2k}(t-s)}{\la^{6-2\be}+1}} 
\,\la^{\be k} \,\er_3^2 \, \la^{(4-2\be)(k-1)} ds \\
= |a_k| e^{-\frac{\la^{2k} t}{\la^{6-2\be}+1}} +
\er_3^2 \, \la^{(2-\be)k} \left(\la^2 + \la^{2\be-4}\right)
\left(1 - e^{-\frac{\la^{2k} t}{\la^{6-2\be}+1}}\right) \\
\le \er_3 \, \la^{(2-\be)k} \left[ \er_3 \left(\la^2 + \la^{2\be-4}\right)
+ \left(1 - \er_3 \left(\la^2 + \la^{2\be-4}\right)\right)
e^{-\frac{\la^{2k}t}{\la^{6-2\be}+1}} \right] .
\end{eqnarray*}
By assumption, we have 
$1 - \er_3 \left(\la^2 + \la^{2\be-4}\right) \ge 0$,
so, the expression in the rectangle brackets 
attains its maximum at $t=0$,
and this maximum is equal to $1$.
Therefore,
$|u_k(t)| \le \er_3 \la^{(2-\be)k}$,
and \eqref{26} is fulfilled for $n=k$.
\end{proof}

Theorem \ref{t26} implies two corollaries.
The first one is the strong solvability of the problem \eqref{01} 
on a small interval of time.

\begin{cor}
\label{c27}
Let $|a_n| \le \er_3 \la^{(2-\be)n}$ for $n \ge K$,
$\er_3 \le \left(\la^2 + \la^{2\be-4}\right)^{-1}$.
Assume moreover, that
$|a_K| < \er_3 \la^{(2-\be)K}$.
Let $\{u_n\}$ be a Galerkin solution to \eqref{01}.
Then there is $\tau > 0$ such that 
\begin{equation}
\label{27}
|u_n (t)| \le \er_3 \la^{(2-\be)n} \quad
\text{for}\ n \ge K, \quad t \in [0,\tau] .
\end{equation}
\end{cor}

\begin{proof}
The function $u_K$ is continuous, 
so there is $\tau > 0$ such that
$$
|u_K (t)| \le \er_3 \la^{(2-\be)K} \quad
\text{for} \ t \in [0,\tau] .
$$
Now we apply Theorem \ref{t26} with $L=K$
on the interval $[0,\tau]$,
and we get \eqref{27} for all $n \ge K$.
\end{proof}

The second corollary of Theorem \ref{t26} guarantees
that the solution is strong whenever
{\it the final data} are non-positive.

\begin{cor}
\label{c28}
Let $|a_n| \le \er_3 \la^{(2-\be)n}$ for $n \ge K$,
$\er_3 \le \left(\la^2 + \la^{2\be-4}\right)^{-1}$.
Let $\{u_n\}$ be a Galerkin solution to \eqref{01} on $[0, T]$.
Assume that 
$\{u_n(T)\} \notin l_2^+$,
where the set $l_2^+$ is defined in \eqref{12}.
Then there is a number $L$ such that
$$
|u_n (t)| \le \er_3 \la^{(2-\be)n} \quad
\text{for}\ n \ge L, \quad t \in [0,T] .
$$
\end{cor}

\begin{proof}
Choose a number $L \ge K$ such that $u_L (T) \le 0$.
By virtue of Lemma \ref{l11} and Corollary \ref{c25} 
we have $a_L \le 0$ and
$$
u_L (t) \in [a_L, 0] \quad \text{for all}\ t \in [0,T] .
$$
Therefore, $|u_L(t)| \le \er_3 \la^{(2-\be)L}$,
$t \in [0, T]$, and the claim follows from Theorem \ref{t26}.
\end{proof}

%%%%%%%%%%%%%%%%%%%%%%%%%%%%%%%%%%%%%%%%%%%%%%%%%%%%%%%%%%%%%%%%%%%%%%%%%

\section{Estimate of the integral $\int_0^\infty u_n^3\, dt$}
The aim of this section is the following estimate of weak solutions
$$
\int_0^\infty u_n (t)^3 dt = O (n \la^{-\be n}), \quad n \to \infty,
$$
for a non-negative case.
Such estimate for a finite interval of time 
is proven in \cite{BM} for the system \eqref{03}
under the assumption that the initial data are non-negative, 
$a_n \ge 0$ for all $n$.
We need this estimate for the system \eqref{01}
and for the initial data from $l_2^+$ 
(all components are non-negative beginning from some number $M$).
The proof is very similar to the proof in \cite{BM},
we give it for the sake of completeness.

We need the following covering lemma,
%In the literature, there are many covering lemmas, 
see for example \cite[Ch. I]{St}.
%We need the following simple version.

\begin{lemma}
\label{l31}
Let $A$ be a measurable subset of $\R$.
Let ${\mathcal I}$ be a family of intervals in $\R$,
lengths of which are bounded.
Assume that for all $a \in A$ there is an interval $(a,b)$ 
from ${\mathcal I}$.
Then we can select from this family ${\mathcal I}$
a disjoint subsequence $\{(a_k, b_k)\}$, finite or countable,
such that
$$
|A| \le 2 \sum_k (b_k - a_k) .
$$
\end{lemma}

\subsection{Lemmas}
Let $\{u_n\}$ be a weak solution to \eqref{01}
with the initial data $\{a_n\}$, $a_n \ge 0$ for $n \ge M$.
By virtue of Lemma \ref{l12} $\{u_n\}$ is a Leray-Hopf
solution, and $\|\{u_n(t)\}\|_{l_2}\le\|\{a_n\}\|_{l_2}$ for all $t$.

Following \cite{BM} we introduce
the sequence of functions
$$
E_n(t):=\sum_{k=1}^n u_k (t)^2 .
$$
Then
\begin{equation}
\label{31}
\dot E_n(t)+2\sum_{k=1}^n \la^{2k} u_k(t)^2   
= - 2\la^{\be(n+1)} u_n(t)^2 u_{n+1}(t).
\end{equation}
In particular,
\begin{equation}
\label{315}
\dot E_n(t) \le 0 \quad\text{if}\ n\ge M-1.
\end{equation}

Let us fix
$y>0$ and $n\ge M$.
We consider a closed set
\begin{equation}
\label{32}
A_n(y):=\{s\ge 0:u_n(s)\ge y\ge u_{n+2}(s) \}.
\end{equation}
For this set we construct a family of intervals ${\mathcal I}$,
and then we will apply Lemma \ref{l31}.
Let $s \in A_n(y)$.
We put
\begin{equation}
\label{33}
t(s):= \min\begin{cases}
\inf\{t>s:u_n(t)<y/2\},\\
\inf\{t>s:u_{n+2}(t)>2y\},\\
s+\frac{2}{\la^{2n+2}+2y\la^{\be(n+2)}}.
\end{cases}
\end{equation}

\begin{lemma}
\label{l32}
Let $s \in A_n(y)$,
and $t>s$ be such that $u_n(t)=y/2$.
Then 
$$
E_n(s)-E_n(t)\ge \frac{3y^2}4.
$$
\end{lemma}

\begin{proof}
By definition of $E_n$'s
$$
E_n(s)-E_n(t)=
E_{n-1}(s)-E_{n-1}(t)+u_n(s)^2-u_n(t)^2
\ge u_n(s)^2-u_n(t)^2
%=u_n(s)^2- \frac{y^2}4
\ge \frac{3y^2}4,
$$
where we used the monotonicity of $E_{n-1}$,
see \eqref{315}.
\end{proof}

\begin{lemma}
\label{l33}
Let $s \in A_n(y)$,
and $t>s$ be such that $u_{n+2}(t)=2y$.
Then 
$$
E_{n+1}(s)-E_{n+1}(t)\ge 3y^2.
$$
\end{lemma}

\begin{proof}
We have      
\begin{eqnarray*}
E_{n+1}(s)-E_{n+1}(t)=
E_{n+2}(s)-E_{n+2}(t)-u_{n+2}(s)^2+u_{n+2}(t)^2\\
\ge u_{n+2}(t)^2-u_{n+2}(s)^2
%=4y^2-u_{n+2}(s)^2
\ge 3y^2.\quad
\qedhere
\end{eqnarray*}
\end{proof}

\begin{lemma}
\label{l34}
Let $s \in A_n(y)$,
and assume that
\begin{equation}
\label{34}
t(s)=s+\frac{2}{\la^{2n+2}+2y\la^{\be(n+2)}}.
\end{equation}
Then
$$
E_n(s)-E_n(t(s))\ge c_1y^2,\quad\text{where}\ \ 
c_1=\min\left(\frac{1}{2\la^2},\frac{1}{64\la^{2\be}}\right)
$$
\end{lemma}

\begin{proof}
Formulas \eqref{33} and \eqref{34} imply the inequalities
$$
u_n(\rho) \ge \frac{y}2,\quad
u_{n+2}(\rho) \le 2y
\quad\text{for all}\quad
\rho\in[s,t(s)].
$$
So, similarly to \eqref{11} and using 
$u_{n+1}(s) \ge 0$, we have 
\begin{eqnarray*}
u_{n+1}(\tau) \ge
\int_s^\tau\la^{\be (n+1)} u_n(\rho)^2
\exp\left(-\int_\rho^\tau (\la^{2n+2} 
+ \la^{\be(n+2)} u_{n+2}(\si)) d\si\right)d\rho \\
\ge\frac{\la^{\be (n+1)}y^2}4
\int_s^\tau\exp\left(-(\la^{2n+2} 
+ \la^{\be(n+2)}2y)(\tau-\rho)\right)d\rho \\
=\frac{\la^{\be (n+1)}y^2}{4(\la^{2n+2} + 2y\la^{\be(n+2)})}
\left[1-\exp\left(-(\la^{2n+2} + 2y\la^{\be(n+2)})(\tau-s)\right)\right]
\end{eqnarray*}
for $\tau\in[s,t(s)]$.
Furthermore, due to \eqref{31}
\begin{eqnarray}
\label{35}
E_n(s)-E_n(t(s))\ge 
2\int_s^{t(s)}\left(\la^{2n}u_n(\tau)^2
+\la^{\be(n+1)}u_n(\tau)^2u_{n+1}(\tau)\right)d\tau\\
\ge\frac{\la^{2n}y^2(t(s)-s)}{2} 
+\frac{\la^{2\be(n+1)}y^4}{8(\la^{2n+2} + 2y\la^{\be(n+2)})} 
\int_s^{t(s)}\left[1-e^{-(\la^{2n+2} 
+ 2y\la^{\be(n+2)})(\tau-s)}\right]d\tau.
\nonumber
\end{eqnarray}
Let us estimate the last integral.
If $\tau>(t(s)+s)/2$, then
$$
\tau-s> \frac{t(s)-s}2
=\frac1{\la^{2n+2} + 2y\la^{\be(n+2)}}
$$
due to \eqref{34}.
Therefore,
$$
1-e^{-(\la^{2n+2} + 2y\la^{\be(n+2)})(\tau-s)}
>1-e^{-1}> \frac12,
$$
and thus,
$$
\int_s^{t(s)}\left[1-e^{-(\la^{2n+2} 
+ 2y\la^{\be(n+2)})(\tau-s)}\right]d\tau
>\frac{t(s)-s}4.
$$
Now, \eqref{35} implies    
\begin{eqnarray*}
E_n(s)-E_n(t(s))\ge 
\frac{\la^{2n}y^2(t(s)-s)}{2} 
+\frac{\la^{2\be(n+1)}y^4(t(s)-s)}{32(\la^{2n+2} + 2y\la^{\be(n+2)})} \\
=\frac{y^2\left(16\la^{4n+2}+32y\la^{2n+\be(n+2)}+y^2\la^{2\be(n+1)}\right)}
{16\left(\la^{2n+2} + 2y\la^{\be(n+2)}\right)^2}
\ge y^2\min\left(\frac{1}{\la^2},\frac{1}{2\la^2},
\frac{1}{64\la^{2\be}}\right).\quad
\qedhere
\end{eqnarray*}
\end{proof}

\subsection{Proof of the estimate}
\begin{lemma}
\label{l35}
Let $y>0$, $n\ge M$.
Let the set $A_n(y)$ be defined by
the formula \eqref{32}.
Then
$$
|A_n(y)|\le\frac{c_2\|a\|_{l_2}^2}{y^3\la^{\be n} + y^2 \la^{2n}},
\quad\text{where}\ \ 
c_2=\frac{4(3+c_1)}{3c_1} 
\max \left(\frac1{\la^2}, \frac1{2\la^{2\be}} \right).
$$
\end{lemma}

\begin{proof}
Let us consider the family      
${\mathcal I}$ of intervals 
${\mathcal I}=\{(s;t(s))\}_{s\in A_n(y)}$,
where $t(s)$ is defined by \eqref{33}.
By virtue of Lemma \ref{l31}
there is a finite or countable sequence
of disjoint intervals $\{(s_k;t_k)\}\subset {\mathcal I}$,
such that
$$
|A_n(y)| \le 2 \sum_k (t_k - s_k) .
$$
Due to Lemmas \ref{l32}, \ref{l33} and
\ref{l34}, the number of disjoint intervals $\{(s_k;t_k)\}$
can not exceed
$$
\frac{\|a\|_{l_2}^2}{c_1y^2}
+\frac{\|a\|_{l_2}^2}{3y^2},
$$
because the functions $E_n$ and $E_{n+1}$
are non-increasing, see \eqref{315},
and $E_n(0) \le E_{n+1}(0) \le \|a\|_{l_2}^2$.
The length of each interval      
$(t_k - s_k) \le 
\frac{2}{\la^{2n+2} + 2y\la^{\be(n+2)}}$,
by \eqref{33}.
So,
$$
|A_n(y)| \le 2 \sum_k (t_k - s_k)\le 
\frac{4(3+c_1)\|a\|_{l_2}^2}
{3c_1y^2(\la^{2n+2} + 2y\la^{\be(n+2)})} .
\quad\qedhere
$$
\end{proof}

\begin{lemma}
\label{l36}
Let $a_n \ge 0$ for $n \ge M$,
$\{u_n\}$ be a weak solution to \eqref{01}.
For $y>0$ we put
$$
B_n(y)= \{s\ge 0:u_n(s)\ge y\}.
$$
Then
$$|B_n(y)|\le\frac{c_3\|a\|_{l_2}^2}{y^3\la^{\be n} + y^2 \la^{2n}},
\quad\text{for}\ \ n \ge M,
$$
where
$c_3=c_2(1-\la^{-2\be})^{-1}$.
\end{lemma}

\begin{proof}
As $u_k(s)\to 0$ when $k\to\infty$,
we have
$$
B_n(y)\subset\bigcup_{k=0}^\infty A_{n+2k}(y).
$$
Therefore,
$|B_n(y)|\le\sum_{k=0}^\infty |A_{n+2k}(y)|$,
and the reference to Lemma \ref{l35}
completes the proof.
\end{proof}

\begin{theorem}
\label{t37}
Let $a_n \ge 0$ for $n \ge M$, and
$\{u_n\}$ be a weak solution to \eqref{01}.
Then
$$
\int_0^\infty u_n(t)^3dt\le
3 c_3 \|a\|_{l_2}^2 \la^{- \be n}
\log \left( \la^{(\be-2)n} \|a\|_{l_2} + 1\right) ,
$$
where the constant $c_3$ is
defined in Lemma \ref{l36}.
\end{theorem}

\begin{proof}
By Lemma \ref{l12} we have
$B_n(y)=\emptyset$ if $y>\|a\|_{l_2}$.
Therefore, due to Lemma \ref{l36}
\begin{eqnarray*}
\int_0^\infty u_n(t)^3dt 
=\int_0^{\|a\|_{l_2}}\left|B_n (y)\right| 3y^2 dy 
\le 3 c_3\|a\|_{l_2}^2 \int_0^{\|a\|_{l_2}}
\frac{dy}{\la^{\be n} y + \la^{2n}} \\
= 3 c_3 \|a\|_{l_2}^2 \la^{- \be n}
\log \left( \la^{(\be-2)n} \|a\|_{l_2} + 1\right) .
\quad \qedhere
\end{eqnarray*}
\end{proof}

%%%%%%%%%%%%%%%%%%%%%%%%%%%%%%%%%%%%%%%%%%%%%%%%%%%%%%%%%%%%%%%%%%%%%%%%%
\section{Uniqueness}
\subsection{Uniqueness of a weak solution for
initial data from $l_2^+$}
  
\begin{lemma}
\label{l41}
Let $a_n \ge 0$ for $n \ge M$. 
Let $\{u_n\}$ and $\{v_n\}$ be two weak solutions to \eqref{01}.
If we set
$$
\psi_N(t)=\sum_{n=1}^N 2^{-n} 
\left(u_n(t)-v_n(t)\right)^2,
$$
then for $N > M$ we have
$$
\dot\psi_N(t)\le2\la^{\be(M+1)}\|a\|_{l_2}\psi_N(t)+
\la^{\be(N+1)}2^{-N}\left(u_N(t)^2v_{N+1}(t)+u_{N+1}(t)v_N(t)^2\right).
$$
\end{lemma}

\begin{proof}
Put $w_n(t) = u_n(t) - v_n(t)$.
We have
\begin{eqnarray*}
\frac12 \sum_{n=1}^N \frac{d}{dt} \left(\frac{w_n^2}{2^n}\right)
+ \sum_{n=1}^N \frac{\la^{2n}}{2^n} w_n^2 
= \sum_{n=1}^N \frac{w_n}{2^n}
\left(\la^{\be n} (u_{n-1}^2 - v_{n-1}^2) - 
\la^{\be (n+1)} (u_n u_{n+1} - v_n v_{n+1})\right) \\
= \sum_{n=1}^N \left(\frac{\la^{\be n}}{2^n} w_{n-1} w_n (u_{n-1}+v_{n-1})
- \frac{\la^{\be (n+1)}}{2^{n+1}} 
w_n (w_n(u_{n+1} + v_{n+1}) + w_{n+1}(u_n + v_n))\right)\\
= - \sum_{n=1}^N \frac{\la^{\be (n+1)}}{2^{n+1}} w_n^2 (u_{n+1} + v_{n+1})
- \frac{\la^{\be (N+1)}}{2^{N+1}} w_N w_{N+1} (u_N + v_N) \\
\le - \sum_{n=1}^M \frac{\la^{\be (n+1)}}{2^{n+1}} w_n^2 (u_{n+1} + v_{n+1})
- \frac{\la^{\be (N+1)}}{2^{N+1}} (u_N^2 - v_N^2) (u_{N+1} - v_{N+1}) .
\end{eqnarray*}
In the last sum we have changed the upper limit $N$ by $M$,
because $u_n, v_n \ge 0$ for $n \ge M$.
Due to Lemma \ref{l12}
$|u_n(t)|, |v_n(t)| \le \|a\|_{l_2}$, 
therefore
$$
\frac12 \dot \psi_N(t) \le
\la^{\be (M+1)} \|a\|_{l_2} \psi_M(t) +
\la^{\be (N+1)} 2^{-N-1} 
\left(u_N(t)^2v_{N+1}(t)+u_{N+1}(t)v_N(t)^2\right).
\quad\qedhere
$$
\end{proof}

\begin{theorem}
\label{t42}
Let $\{a_n\} \in l_2^+$, 
where the set $l_2^+$ is defined by \eqref{12}.
Then the weak solution to \eqref{01} is unique.
\end{theorem}

\begin{proof}
Let $\{u_n\}$ and $\{v_n\}$ be weak solutions to \eqref{01}.
Put
$$
\psi_N(t)=\sum_{n=1}^N 2^{-n} 
\left(u_n(t)-v_n(t)\right)^2, \quad
\psi_\infty (t)=\sum_{n=1}^\infty 2^{-n} 
\left(u_n(t)-v_n(t)\right)^2 .
$$
Note, that the sequence $\psi_n$ tends to $\psi_\infty$ 
uniformly on $[0, \infty)$ due to the fact that all the functions
$\{u_n\}$, $\{v_n\}$ are uniformly bounded by $\|a\|_{l_2}$.

We have $a_n \ge 0$ for $n\ge M$ for some $M$.
Lemma \ref{l41} implies that for $N > M$
\begin{eqnarray}
\label{41}
\psi_N (t) \le 2 \la^{\be(M+1)} \|a\|_{l_2} \int_0^t \psi_N (\tau) d\tau \\
+ 2^{-N} \la^{\be(N+1)} \int_0^t \left(u_N (\tau)^3 + u_{N+1} (\tau)^3 +
v_N (\tau)^3 + v_{N+1} (\tau)^3\right) d\tau .
\nonumber
\end{eqnarray}
Due to Theorem \ref{t37} the last integral in \eqref{41}
is $O(N \la^{-\be N})$.
Passing in \eqref{41} to the limit $N\to\infty$, we obtain
\begin{equation}
\label{42}
\psi_\infty (t) \le 2 \la^{\be(M+1)} \|a\|_{l_2} 
\int_0^t \psi_\infty (\tau) d\tau .
\end{equation}
This inequality yields $\psi_\infty \equiv 0$,
and $u_n \equiv v_n$.
\end{proof}

In this proof we followed almost literally \cite{BM},
where such uniqueness is established for the system \eqref{03}.
The difference is that in \cite{BM} all $a_n$ are supposed to be 
non-negative, whereas we suppose that $a_n \ge 0$ beginning from 
some $M$.
It leads to the appearance of 
the first term in the right hand side of \eqref{41},
and of the non-trivial right hand side in \eqref{42}.
Nevertheless, \eqref{42} implies that the function 
$\psi_\infty$ is identically zero. 

\subsection{Non sign-definite case}
{\it Proof of Theorem \ref{t07}.}
The existence of a Leray-Hopf solution is proved in Theorem \ref{t22}.
Let us establish its uniqueness.
Let $\{u_n\}$ be a Galerkin solution, and $\{v_n\}$ be 
a Leray-Hopf solution to \eqref{01}, 
see Definitions \ref{o21} and \ref{o02}.
It is sufficient to prove that $u_n \equiv v_n$.
Let $\er_0 < \er$, where $\er$ is defined in \eqref{05},
and assume that
$$
|a_n| \le \er_0 \la^{(2-\be)n} \quad \text{for}\ \ n \ge K.
$$
Let us define $t_0$ as %We introduce into consideration the time
\begin{equation*}
t_0 = \begin{cases}
\inf \{t : \{u_n(t)\} \in l_2^+\}, \quad 
\text{if this set is non-empty}, \\
+ \infty, \quad \text{if} \ \ \{u_n(t)\} \notin l_2^+ 
\ \ \text{for all} \ t ;
\end{cases}
\end{equation*}
the set $l_2^+$ is defined in \eqref{12}.
Assume first that $0<t_0<\infty$.

There are two possibilities.

1) Let $\{u_n(t_0)\}\in l_2^+$.
For any $t_1 < t_0$ we have $\{u_n(t_1)\}\notin l_2^+$,
and by virtue of Corollary \ref{c28} there is a number $L$ such that
$$
|u_n (t)| \le \er_0 \la^{(2-\be)n} \quad
\text{for all}\ n \ge L, \quad t \in [0,t_1] .
$$
Therefore, $u_n \equiv v_n$ on $[0,t_1]$ due to Theorem \ref{t14}.
It is true for all $t_1 < t_0$, and thus
$u_n \equiv v_n$ on $[0,t_0]$, as all these functions are continuous.
In particular, 
$u_n(t_0) = v_n(t_0)$, and $\{u_n(t_0)\}\in l_2^+$.
Now, we can apply Theorem \ref{t42} on the interval $[t_0, \infty)$,
and therefore $u_n \equiv v_n$ on $[t_0, \infty)$ as well.

2) Let $\{u_n(t_0)\}\notin l_2^+$.
In this case we can apply the Corollary \ref{c28} on $[0,t_0]$,
which means that there is a number $L$ such that
$$
|u_n (t)| \le \er_0 \la^{(2-\be)n} \quad
\text{for all}\ n \ge L, \quad t \in [0,t_0] .
$$
In particular, 
$|u_n(t_0)| \le \er_0 \la^{(2-\be)n}$, $n \ge L$.
By Corollary \ref{c27} there exists a time $t_2 > t_0$
such that there is a Galerkin solution $\{\tilde u_n(t)\}$
to the problem \eqref{01} on the interval $[t_0, t_2]$
with the initial data $\tilde u_n(t_0) = u_n(t_0)$,
and moreover,
$$
|\tilde u_n (t)| \le \er \la^{(2-\be)n} \quad
\text{for}\ n \ge L, \quad t \in [t_0, t_2] .
$$
Furthermore, $\{\tilde u_n\}$ is a Leray-Hopf solution on $[t_0, t_2]$,
because any Galerkin solution is a Leray-Hopf solution,
see subsection \ref{s21}.
We put
\begin{equation*}
\hat u_n (t) = \begin{cases}
u_n(t), \quad t \le t_0, \\
\tilde u_n(t), \quad t_0 < t \le t_2 .
\end{cases}
\end{equation*}
This $\{\hat u_n\}$ is a Leray-Hopf solution to \eqref{01}
on the interval $[0,t_2]$, and
$$
|\hat u_n (t)| \le \er \la^{(2-\be)n} \quad
\text{for}\ n \ge L, \quad t \in [0, t_2] .
$$
By Theorem \ref{t14} $\{\hat u_n\}$ coincides with 
$\{u_n\}$ and with $\{v_n\}$,
$\hat u_n \equiv u_n \equiv v_n$, on $[0, t_2]$.
By definition of the time $t_0$ we have
$\{u_n(t_2)\}\in l_2^+$, see Lemma \ref{l11}.
Now, $u_n \equiv v_n$ also on $[t_2, \infty)$ due to Theorem \ref{42}.

The cases $t_0 = 0$ and $t_0 = \infty$ can be treated in the same way.
\qed

%%%%%%%%%%%%%%%%%%%%%%%%%%%%%%%%%%%%%%%%%%%%%%%%%%%%%%%%%%%%%%%%%%%%%%%%%
\section{Stationary solutions}
In this section we study the equation \eqref{06} and prove Theorem \ref{t010}.
There are no differential equations here, 
the theory of number sequences only.

\subsection{Case $\be < 3$}
Following \cite{BMR} we change the variables
\begin{equation}
\label{51}
a_n = - \la^{(2-\be)n - 2} b_n .
\end{equation}
Then \eqref{06} is equivalent to the system
\begin{equation}
\label{52}
b_n b_{n+1} = b_n + u b_{n-1}^2, \quad b_0 = 0, \ \ 
n = 1, 2, \dots ,
\end{equation}
where $u = \la^{2\be-6}$.
In this subsection we show that if $\la^{2\be-6} < 1$
(it is equivalent to the assumption $\be < 3$)
then there exists a sequence of positive numbers $\{b_n\}$ 
satisfying \eqref{52}.
In order to satisfy \eqref{52} for $n=1$ we take $b_2=1$.

One can consider \eqref{52} as a recurrent formula
which allows to express $b_{n+1}$ via $b_{n-1}$ and $b_n$.
It turns out that it is more convenient to make the process
run backward, and find $b_{n-1}$ via $b_n$ and $b_{n+1}$,
$b_{n-1} = \sqrt{b_n (b_{n+1} - 1) / u}$.
We construct an auxiliary number sequence $\{c_n\}$,
which is defined by two first terms $c_0, c_1 > 1$, 
and by the following rule:
\begin{equation}
\label{53}
\text{if} \ c_{k-1} > 1, \quad
\text{then} \ 
c_{k+1} = \sqrt{\frac{c_k\left(c_{k-1}-1\right)}{u}} ;
\end{equation}
if $c_{k-1} \le 1$, then the sequence stops at $c_k$.

Intruduce the notations
\begin{equation}
\label{54}
\ka = \frac12 + \sqrt{\frac14 + \frac1u}, \quad
\nu = \frac14\left(1 + \sqrt{1+\frac8u}\right) .
\end{equation}
It is easy to see that 
$\ka > \nu > 1$ if $u < 1$.

\begin{lemma}
\label{l51}
Let $\de > 0$ be such that $\ka \de < \frac{u}{1-u}$.
Assume that
$$
c_0 \ge \frac1{1-u} - \de, \quad
c_1 \ge \frac1{1-u} - \ka \de 
$$
(clearly, $c_0, c_1 > 1$).
Then
$$
c_2 \ge \frac1{1-u} - \ka^2 \de .
$$
\end{lemma}

\begin{proof}
We have
\begin{eqnarray*}
c_2 = \frac1{\sqrt{u}} \sqrt{c_1 (c_0-1)} \ge
%\frac1{\sqrt{u}} \sqrt{\left(\frac1{1-u} - \ka \de \right)
%\left(\frac{u}{1-u} - \de \right)} \\ =
\frac{1}{1-u} \sqrt{\left(1 - \ka \de (1-u)\right)
\left(1 - \frac{\de}{u} (1-u)\right)} \\
> \frac1{1-u} \left(1 - \ka \de (1-u)\right)
\left(1 - \frac\de{u} (1-u)\right) 
> \frac1{1-u} - \ka \de - \frac\de{u} 
= \frac1{1-u} - \ka^2 \de 
\end{eqnarray*}
by definition of $\ka$, see \eqref{54}.
\end{proof}

By induction we obtain

\begin{cor}
\label{c52}
Assume that $\ka^n \de < \frac{u}{1-u}$,
$$
c_0 \ge \frac1{1-u} - \de, \quad
c_1 \ge \frac1{1-u} - \ka \de .
$$
Then the sequence $\{c_k\}$ is defined at least until $k=n+1$, and 
$c_n \ge \frac1{1-u} - \ka^n \de$.
\end{cor}

\begin{lemma}
\label{l53}
Let $\de > 0$ be such that $\nu \de < \frac{u}{1-u}$,
$$
1 < c_0 \le \frac1{1-u} - \de, \quad
1< c_1 \le \frac1{1-u} - \nu \de .
$$
Then
$$
c_2 \le \frac1{1-u} - \nu^2 \de .
$$
\end{lemma}

\begin{proof}
We have
\begin{eqnarray*}
c_2 \le \frac{1}{1-u} \sqrt{\left(1 - \nu \de (1-u)\right)
\left(1 - \frac{\de}{u} (1-u)\right)} \\
\le \frac1{2(1-u)} \left(1 - \nu \de (1-u) +
1 - \frac\de{u} (1-u)\right) =
\frac1{1-u} - \frac12 \left(\nu + \frac1{u}\right) \de  
= \frac1{1-u} - \nu^2 \de 
\end{eqnarray*}
by definition of $\nu$, see \eqref{54}.
\end{proof}

\begin{cor}
\label{c54}
Let $\nu^m \de < \frac{u}{1-u}$,
$c_0 \le \frac1{1-u} - \de$,
$c_1 \le \frac1{1-u} - \nu \de$.
Assume that the sequence $\{c_k\}$ is defined until $k=m$,
i.e. $c_k > 1$ for $k = 0, 1, \dots, m-1$.
Then 
$$
c_m \le \frac1{1-u} - \nu^m \de .
$$
\end{cor}

\begin{lemma}
\label{l55}
Let $M \in \N$.
There exists a positive number $\de_1$ such that if
$$
c_0 = \frac1{1-u} - \de_1, \quad
c_1 = \frac1{1-u} - \nu \de_1 ,
$$
where $\nu$ is defined in \eqref{54},
then the sequence $\{c_k\}$ is defined until $k=n+1$
with some $n > M$, and
$$
c_k < \frac1{1-u} \ \ \forall k, \quad
\text{and} \ \ c_n < 1.
$$
\end{lemma}

\begin{proof}
Choose $\de_1$ such that 
$\ka^M \de_1 < \frac{u}{1-u}$.
Then by Corollary \ref{c52}
$c_0, c_1, \dots, c_M > 1$.
On the other hand, Corollary \ref{c54} implies
that the sequence $\{c_k\}$ can not be infinite.
So, there is $n > M$ such that $c_n < 1$.
The inequality $c_k < \frac1{1-u}$ for all $k$
follows also from Corollary \ref{c54}.
\end{proof}

\begin{cor}
\label{c56}
Let $M \in \N$.
There exists a positive number $\de_0$ such that if
$$
c_0 = \frac1{1-u} - \de_0, \quad
c_1 = \frac1{1-u} - \nu \de_0 ,
$$
then the sequence $\{c_k\}$ is defined until $k=n+1$
with some $n > M$, and
$$
c_k < \frac1{1-u} \ \ \forall k, \quad
\text{and} \ \ c_n = 1.
$$
\end{cor}

\begin{proof}
Put 
$c_0 = \frac1{1-u} - \de$, $c_1 = \frac1{1-u} - \nu \de$
with $\de \in [0, \de_1]$. 
%where $\de_1$ is the number from Lemma \ref{l55}.
We consider the elements of the sequence $\{c_k\}_{k=0}^{n+1}$ 
as the functions of $\de \in [0, \de_1]$.
Here $\de_1$ and $n$ are the numbers from Lemma \ref{l55}.
Clearly, all the functions $c_k (\de)$ are continuous
and decreasing in $\de$.
Note, that
$c_k (0) = \frac1{1-u}$ for all $k$,
and
$$
c_n (\de_1) < 1 < \frac1{1-u} = c_n (0) .
$$
Therefore, there exists $\de_0$ such that $c_n (\de_0) = 1$.
\end{proof}

\begin{cor}
\label{c57}
Let $M \in \N$.
There is a number $n>M$, and a positive number $d_1$,
such that the sequence $\{d_k\}$, defined by two first terms
$\{d_1, d_2\}$ with $d_2 = 1$, and by the recurrent formula
\begin{equation}
\label{55}
d_{k+1} = 1 + \frac{u d_{k-1}^2}{d_k} ,
\end{equation}
obeys the property 
$$
d_k < \frac1{1-u} \quad \text{for} \ \ k \le n+2 .
$$
\end{cor}

\begin{proof}
It is sufficient to take
$$
d_k = c_{n+2-k}, \quad k = 1, \dots, n+2,
$$
where $\{c_k\}$ is the sequence from the precedent Corollary.
\end{proof}

\begin{theorem}
\label{t58}
Let $u < 1$.
Then there is an infinite number sequence $\{b_k\}$,
$b_k > 0$, $b_2 = 1$, satisfying \eqref{52} and such that
$b_k \le \frac1{1-u}$ for all $k$.
\end{theorem}

\begin{proof}
By virtue of Corollary \ref{c57} there are a sequence
of numbers  $\{n_m\}$, $n_m \to \infty$, and a sequence of sequences
$\{d_k^{(m)}\}_{k=1, \dots, n_m}$, such that
$$
d_2^{(m)} = 1, \quad 
d_{k+1}^{(m)} = 1 + \frac{u \left(d_{k-1}^{(m)}\right)^2}{d_k^{(m)}} ,
\quad d_k^{(m)} < \frac1{1-u} .
$$
The sequence of the first elements $\{d_1^{(m)}\}_{m=1}^\infty$
is bounded.
Without loss of generality one can assume
that this sequence converges, $d_1^{(m)} \underset{m\to\infty}\longrightarrow b_1$.
%\to_{m\to\infty} b_1$.
Each element $d_k^{(m)}$ is a continuous function of $d_1^{(m)}$.
Therefore, $d_k^{(m)} \underset{m\to\infty}\longrightarrow b_k$,
%\to_{m\to\infty} b_k$,
and it is clear that
$$
b_{k+1} = 1 + \frac{u b_{k-1}^2}{b_k} ,
\quad \text{and} \ \ 
b_k \le \frac1{1-u} \quad \text{for all}\ \ k . \quad
\qedhere
$$
\end{proof}

Taking into account the change \eqref{51} 
we see that Theorem \ref{t58} means
the existence of non-trivial solution $\{a_n\}$ 
to the system \eqref{06} such that 
$$
|a_n| \le \frac{\la^{(2-\be)n}}{(1-u)\la^2} .
$$
The estimate $|a_n| \ge \la^{(2-\be)n-2}$
for any non-trivial solution is done in Theorem \ref{t09}.
Thus, the case $\be < 3$ of Theorem \ref{t010} is completely proven.

\begin{rem}
\label{r59}
The sequence $\{b_k\}$ from Theorem \ref{t58} converges,
$b_k \to \frac1{1-u}$.
Indeed, the sequence $\{c_k\}$ from Corollary \ref{c56} decreases
(the inequality $c_0>c_1>c_2$ implies $c_2>c_3$,
and one can proceed by induction).
So, the sequence $\{b_k\}$ increases, and therefore converges,
$b_k \to b_\infty$.
Finally, \eqref{52} yields $b_\infty = \frac1{1-u}$.
\end{rem}

\subsection{Case $\be = 3$}
In this subsection we study the equation \eqref{52} with $u=1$.

\begin{lemma}
\label{l510}
Let $b_1, b_2 > 0$, $b_{n+1} = 1 + b_{n-1}^2 b_n^{-1}$.
Then the sequence $\{b_n\}$ is unbounded.
\end{lemma}

\begin{proof}
Put
$$
\limsup_{k\to\infty} b_{2k+1} = K, \quad
\limsup_{k\to\infty} b_{2k} = L .
$$
Suppose that $K, L < \infty$.
Without loss of generality one can assume $K \ge L$.
Let us pick $\er > 0$ such that
$$
1 + \frac{(K-\er)^2}{L+\er} > K + \er .
$$
By definition of $K$ and $L$,  
we can choose a number $N$ such that
$$
b_{2k} \le L + \er, \quad b_{2k+1} \le K + \er,
\quad \text{for all} \ \ k \ge N,
$$
and a number $M>N$ such that $b_{2M-1} > K - \er$.
Then we have
$$
b_{2M+1} = 1 + \frac{b_{2M-1}^2}{b_{2M}} \ge 
1 + \frac{(K-\er)^2}{L+\er} > K + \er ,
$$
which is a contradiction.
\end{proof}

Now we consider the sequence $\{a_n\}$, see \eqref{51}.
Several first terms can vanish,
but if there is a non-zero term in the sequence,
then the Lemma \ref{l510} means that
$$
\limsup_{n\to\infty} a_n \la^n = + \infty
\quad \text{for} \ \ \be = 3 .
$$

In the rest of this subsection we show that there is 
a sequence satisfying \eqref{52} with $u=1$,
which increases like a linear function.
Again, we consider an auxiliary sequence $\{c_n\}$,
defined by the rule:
\begin{eqnarray*}
\text{if} \ c_{k-1} > 1, \quad
\text{then} \ 
c_{k+1} = \sqrt{c_k\left(c_{k-1}-1\right)} ,\\
\text{if}\ c_{k-1} \le 1, \quad
\text{then the sequence stops at} \ c_k.
\end{eqnarray*}

\begin{lemma}
\label{l511}
If $c_0 \le A$, $c_1 \le A-1/3$,
then $c_k \le A - k/3$.
\end{lemma}

\begin{proof}
Clearly,
$c_2 \le \sqrt{(A-1/3)(A-1)} \le A - 2/3$.
Further we proceed by induction.
\end{proof}

\begin{lemma}
\label{l512}
Let $c_0 = A$, $c_1 = A-1/3$.
Then the inequalities
\begin{equation}
\label{56}
c_{k+1} < c_k < c_{k+1} + 1,
\end{equation}
\begin{equation}
\label{57}
c_k \ge A - k,
\end{equation}
are fulfilled.
\end{lemma}

\begin{proof}
As $c_2 \le A - 2/3 < c_1$,
we obtain by induction again $c_{k+1} < c_k$.
On the other hand, this inequality means
$$
c_{k+1} = \sqrt{c_k\left(c_{k-1}-1\right)} < c_k 
\quad \Rightarrow \quad
c_{k-1}-1 < c_k,
$$ 
and \eqref{56} is proved.
The estimate \eqref{57} follows from \eqref{56}.
\end{proof}

\begin{lemma}
\label{l513}
Let $M \in \N$, $A_1 > M+1$.
If $c_0 = A_1$, $c_1 = A_1-1/3$,
then the sequence $\{c_k\}$ is defined until 
$k=n+1$, $n>M$, and $c_n < 1$.
\end{lemma}

\begin{proof}
Due to \eqref{57} $c_M \ge A_1 - M > 1$, 
where from $n>M$.
By virtue of Lemma \ref{l511},
there exists $n$ such that $c_n < 1$.
\end{proof}

The proofs of the following Corollaries \ref{c514}, \ref{c515}
and Theorem \ref{t516} are similar to the proofs of 
Corollaries \ref{c56}, \ref{c57} and Theorem \ref{t58} respectively.

\begin{cor}
\label{c514}
Let $M \in \N$. 
There is $A_0$ such that if
$c_0 = A_0$, $c_1 = A_0-1/3$,
then the sequence $\{c_k\}$ is defined until 
$k=n+1$, $n>M$, and $c_n = 1$.
\end{cor}

\begin{cor}
\label{c515}
Let $M \in \N$. 
There is $n > M$ and $d_1 \in (0,1)$ such that 
the sequence $\{d_k\}$ defined by two first elements
$\{d_1, d_2\}$, $d_2=1$, and by the recurrent formula
$d_{k+1} = 1 + d_{k-1}^2 d_k^{-1}$
satisfies the condition $d_k < k$ for $k \le n$.
\end{cor}

In the proof of this Corollary one uses 
the inequality \eqref{56}.

\begin{theorem}
\label{t516}
There exists an infinite sequence $\{b_k\}$,
$b_k > 0$, $b_2 = 1$, such that
$$
b_{k+1} = 1 + \frac{b_{k-1}^2}{b_k} \quad
\text{and} \quad b_k \le k \ \ \text{for all} \ \ k .
$$
\end{theorem}

Thus, Theorem \ref{t010} in the part $\be = 3$ is proved.

\subsection{Case $\be > 3$}
Let us do the change 
$a_n = - e_n \la^{-\be n/3}$ in \eqref{06}.
Then \eqref{06} is equivalent to the system
$$
e_n e_{n+1} = \la^{((6-2\be)n-2\be)/3} e_n + e_{n-1}^2 .
$$
By Theorem \ref{t09} $e_n \ge 0$, therefore
$e_{n-1} \le \max(e_n, e_{n+1})$.
Iterating this inequality we get
$$
e_{n-1} \le \max(e_n, e_{n+1}) \le \max(e_{n+1}, e_{n+2})
\le \dots \le \max(e_m, e_{m+1})
$$
for all $m \ge n$.
So, $\limsup_{n\to\infty} e_n > 0$
unless all $e_k = 0$.
We proved

\begin{lemma}
\label{l517}
Let $\be > 3$, $\{a_n\}$ be a non-trivial solution to \eqref{06}.
Then
$$
\limsup_{n\to\infty} (-a_n) \la^{\frac{\be n}3} > 0.
$$
\end{lemma}

Let us show that decreasing like $\la^{-\be n/3}$ is possible.
We return to the equation \eqref{52}, 
and consider the auxiliary sequence \eqref{53};
now $u > 1$.

\begin{lemma}
\label{l518}
Let $c_0 = A$, $c_1 = Au^{-1/3}$.
Then we have the following relations:

a) $c_2 < A u^{-2/3}$;

b) $c_{k+1} < c_k u^{-1/3}$;

c) $c_{k-1} < c_k u^{1/3} + 1$;

d) $c_k > (A-1) u^{-k/3} - (u^{1/3}-1)^{-1}$.
\end{lemma}

\begin{proof}
a) Clearly,
$$
c_2 = \sqrt{c_1(c_0-1)u^{-1}} =
\sqrt{A(A-1)u^{-4/3}} < Au^{-2/3} .
$$

b) follows from a) by induction.

c) We have 
$$
c_{k+1} = \sqrt{c_k(c_{k-1}-1)u^{-1}}  < c_k u^{-1/3}
$$
due to the point b). Therefore,
$c_{k-1}-1 < c_k u^{1/3}$.

d) follows from c) by induction.
\end{proof}

\begin{cor}
\label{c519}
Let $M \in \N$.
There exists a number $A_0$ such that
if $c_0 = A_0$, $c_1 = A_0 u^{-1/3}$,
then the sequence $\{c_k\}$ is defined until 
$k=n+1$, $n>M$, and $c_n = 1$.
\end{cor}

The proof is similar to the proof of Corollary \ref{c56}.

\begin{lemma}
\label{l520}
Let $M \in \N$. 
There exist a number $n > M$ and a number $d_1 \in (0,1)$ such that 
the sequence $\{d_k\}$ defined by two first elements
$\{d_1, d_2\}$, $d_2=1$, and by the formula \eqref{55} with $u>1$,
satisfies the condition 
\begin{equation}
\label{58}
d_{k+1} \le \frac{u^{k/3}-1}{u^{1/3}-1} \quad
\text{for} \ \ 1 \le k \le n+2 .
\end{equation}
\end{lemma}

\begin{proof}
We put $d_k = c_{n+2-k}$, where $\{c_k\}$ is the sequence
from the Corollary \ref{c519}.
Due to Lemma \ref{l518} c) we have
$d_{k+1} < d_k u^{1/3} + 1$.
Taking into account that $d_2=1$, we get from here \eqref{58}.
\end{proof}

In the same manner as in the proof of the Theorem \ref{t58},
we get from here

\begin{theorem}
\label{t521}
Let $u>1$.
There exists an infinite sequence $\{b_k\}_{k=1}^\infty$,
$b_k > 0$, $b_2 = 1$, satisfying \eqref{52} and such that
$$
b_{k+1} \le \frac{u^{k/3}-1}{u^{1/3}-1}
\quad \text{for all} \ \ k.
$$
\end{theorem}

By virtue of \eqref{51} this Theorem means that 
there is a non-trivial solution $\{a_n\}$ to the system \eqref{06}
such that $a_n = O (\la^{-\be n/3})$, $n\to \infty$.
Thus, the part $\be > 3$ is done, and so,
Theorem \ref{t010} is completely proven.

\begin{rem}
\label{r522}
In \cite{BFM} the Theorem \ref{t521} is proved,
and moreover, the existence of a positive limit 
$\lim_{n\to\infty} b_n u^{-n/3}$ is shown.
We provided our proof for the sake of completeness,
as it is very similar to the cases $\be<3$, $\be=3$.
In \cite{BFM}, the equation \eqref{52} appears in studying 
of self-similar solution to the problem \eqref{03}.
Let $\{b_n\}$ be a solution to \eqref{52} with
$u=\ka^2 > 1$.
Then the functions
$$
u_n (t) = - \, \frac{b_n \ka^{-n}}{1-t}
$$
solve the problem \eqref{03} with initial data
$u_n(0) = - b_n \ka^{-n}$.
If $b_n = O(u^{n/3}) = O(\ka^{2n/3})$, $n\to \infty$,
then the initial data $u_n(0) = O(\ka^{-n/3})$
are good enough.
Nevertheless, any norm of solution $\{u_n\}$
becomes infinite when $t\to 1$.
\end{rem}

%%%%%%%%%%%%%%%%%%%%%%%%%%%%%%%%%%%%%%%%%%%%%%%%%%%%%%%%%%%%%%%%%%%%%%%%%

\end{document}